\documentclass[12pt,oneside,reqno]{amsart}
\usepackage{amsmath,amsfonts,amssymb}

\theoremstyle{plain}
\newtheorem{theorem}{Theorem}[section]
\newtheorem{corollary}{Corollary}[theorem]

\theoremstyle{definition}

\theoremstyle{remark}

\def\customauthor{\empty}
\def\customdate{\empty}
\let\oldauthor\author
\renewcommand{\author}[1]{\def\customauthor{#1}}
\renewcommand{\date}[1]{\def\customdate{#1}}
\oldauthor{\customauthor\\\customdate}

\begin{document}

\overfullrule=0pt

\author{ Dirk Veestraeten\\\\\\\\}
\date{\today\\\\}

\address{Amsterdam School of Economics \\
University of Amsterdam\\
Roetersstraat 11\\
1018WB Amsterdam\\
 the Netherlands} 
 
\email{d.j.m.veestraeten@uva.nl}

\title[Integral transforms and integral representations \dots ] 
      {Some Laplace transforms and integral representations for parabolic cylinder functions and error functions}

\begin{abstract}
This paper uses the convolution theorem of the Laplace transform to derive new inverse Laplace transforms for the product of two parabolic cylinder functions in which the arguments may have opposite sign. These transforms are subsequently specialized for products of the error function and its complement thereby yielding new integral representations for products of the latter two functions. The transforms that are derived in this paper also allow to correct two inverse Laplace transforms that are widely reported in the literature and subsequently uses one of the corrected expressions to obtain two new definite integrals for the generalized hypergeometric function.
\end{abstract}

\maketitle

\noindent  \\

\noindent  \\

\section{Introduction}

\noindent The parabolic cylinder function is intensively used in various domains such as chemical physics \cite{kcw96}, lattice field theory \cite{dfg92}, astrophysics \cite{zm07}, finance \cite{nl19}, neurophysiology \cite{cr71} and estimation theory \cite{bv09}. Products of parabolic cylinder functions involving both positive and negative arguments arise in, for instance, problems of condensed matter physics \cite{cd83,m99} and the study of real zeros of parabolic cylinder functions \cite{d75,em99,em08}.

However, the extensive tables of inverse Laplace transforms \cite{emoti154,ob73,pbm592} present relatively few expressions for products of parabolic cylinder functions especially when signs of the arguments differ. For example, \cite{pbm592} only specifies the following inverse Laplace transforms for such set--up
\begin{align*}
&  D_{\nu}\left(  a\sqrt{p+\sqrt{p^{2}+b^{2}}}\right)  \left\{  D_{\nu}\left(
-a\sqrt{\sqrt{p^{2}+b^{2}}-p}\right)  \pm D_{\nu}\left(  a\sqrt{\sqrt
{p^{2}+b^{2}}-p}\right)  \right\}  \\
&  D_{\nu}\left(  a\sqrt{p+\sqrt{p^{2}-b^{2}}}\right)  \left\{  D_{\nu}\left(
-a\sqrt{p-\sqrt{p^{2}-b^{2}}}\right)  \pm D_{\nu}\left(  a\sqrt{p-\sqrt
{p^{2}-b^{2}}}\right)  \right\}
\end{align*}
see Equations (3.11.4.9) and (3.11.4.10).

This paper uses the convolution theorem of the Laplace transform to derive inverse Laplace transforms for
\begin{align*}
& p^{i}\exp\left(  \tfrac{1}{2}p\left( y-x\right)  \right)  D_{\mu}\left(
2^{1/2}y^{1/2}p^{1/2}\right)  \left\{  D_{\nu}\left(  -2^{1/2}x^{1/2}%
p^{1/2}\right)  \pm D_{\nu}\left(  2^{1/2}x^{1/2}p^{1/2}\right)  \right\}
\end{align*}
with $i=0$ or $-\tfrac{1}{2}$, i.e. for expressions in which the arguments have opposite sign and differ, and where also the orders take on different values.

These results also offer inverse Laplace transforms for the product of (complementary) error functions as the parabolic cylinder function for order $-1$ specializes into the complementary error function. As a result, novel integral representations are obtained for products of the (complementary) error functions and, for instance, the integral representation for 1 -- erf$(a)^{2}$ in \cite{ng69} can be generalized into 1 -- erf$(a)$erf$(b)$.

The paper also corrects two inverse Laplace transforms that are reported in \cite{emoti154,ob73,pbm592}. Combinations of one of the corrected results with the results derived in this paper are particularly interesting as they yield two definite integrals for the generalized hypergeometric function that are not reported in, for instance, the comprehensive overview in \cite{gr14}.

The remainder of this paper is organized as follows. Section $2$ presents the relation between the parabolic cylinder function and the Kummer confluent hypergeometric function that is central to the subsequent derivations. Also, more detail is presented on the formulation of the convolution theorem for the Laplace transform given that the limits of integration in the integrals in the product differ. Section $3$ presents the inverse Laplace transforms for products of the parabolic cylinder function and uses these results to obtain novel integral representations for products of (complementary) error functions. Section $4$ corrects two widely-reported inverse Laplace transforms. Section $5$ uses one of these corrected expressions together with the results of Section $3$ to derive two novel definite integrals for the generalized hypergeometric function.
\medskip
\section{Notation and background} 
\noindent The parabolic cylinder function in the definition of Whittaker \cite{w02} is denoted 
by $D_{\nu }\left( z\right) $, where $\nu $ and $z$ represent the order and the argument, respectively. Equation (4) on p. 117 in \cite{emoth253}
defines the parabolic cylinder function in terms of Kummer's confluent
hypergeometric function $\Phi \left( a;b;z\right) $ as follows%
\begin{align}
& D_{\nu }\left( z\right) = 2^{\nu /2}\exp \left( -\tfrac{1}{4}z^{2}\right) %
\left\{ \dfrac{\Gamma \left[ 1/2\right] }{\Gamma \left[ \left( 1-\nu \right)
/2\right] }\Phi \left( -\dfrac{\nu }{2};\dfrac{1}{2};\dfrac{1}{2}%
z^{2}\right) \right.  \nonumber \\
& \hspace{2cm} \left. +\dfrac{z}{2^{^{1/2}}}\dfrac{\Gamma \left[ -1/2\right] }{\Gamma
\left[ -\nu /2\right] }\Phi \left( \dfrac{1-\nu }{2};\dfrac{3}{2};\dfrac{1}{2%
}z^{2}\right) \right\}  \label{dkummer}
\end{align}
\noindent where $\Gamma \left[ \nu \right] $ denotes the gamma function. Note that the definition (\ref{dkummer}) holds for $z$ as well as $-z$ and adding the corresponding relation for $D_{\nu }\left( -z\right) $ to (\ref
{dkummer}) then gives
\begin{align}
& D_{\nu }\left( -z\right) -D_{\nu }\left( z\right) = \dfrac{z2^{\left( \nu
+3\right) /2}\sqrt{\pi }}{\Gamma \left[ -\nu /2\right] }\exp \left( -\tfrac{1%
}{4}z^{2}\right) \Phi \left( \dfrac{1-\nu }{2};\dfrac{3}{2};\dfrac{1}{2}%
z^{2}\right)  \label{sum1} \\
& D_{\nu }\left( -z\right) +D_{\nu }\left( z\right) = \dfrac{2^{\left( \nu
+2\right) /2}\sqrt{\pi }}{\Gamma \left[ \left( 1-\nu \right) /2\right] }\exp
\left( -\tfrac{1}{4}z^{2}\right) \Phi \left( -\dfrac{\nu }{2};\dfrac{1}{2};%
\dfrac{1}{2}z^{2}\right) \label{sum2}
\end{align}%
\noindent see Equations (46:5:4) and (46:5:3) in \cite{oms09}.

The convolution theorem of the Laplace transform will be used to derive inverse Laplace transforms
for products of two parabolic cylinder functions. The functions in the products are taken from inverse Laplace
transforms for the parabolic cylinder function and the Kummer confluent hypergeometric function, respectively. The inverse Laplace transforms that will be used for $\Phi \left( a;b;z\right) $ and $D_{\nu }\left( z\right) $ are not both defined over the half--line $\left( 0,\infty \right)$. As a result, the convolution theorem becomes somewhat more involved. The Laplace transforms of the original functions $f_{1}\left( t\right) $ and $f_{2}\left( t\right) $ are defined as
\begin{align}
& \overline{f}_{1}\left( p\right)  = \int_{\alpha _{1}}^{\beta _{1}}\exp
\left( -pt\right) f_{1}\left( t\right) dt\hspace{1.5cm}\beta _{1}>\alpha
_{1} \nonumber \\
& \overline{f}_{2}\left( p\right)  = \int_{\alpha _{2}}^{\beta _{2}}\exp
\left( -pt\right) f_{2}\left( t\right) dt\hspace{1.5cm}\beta _{2}>\alpha
_{2}\nonumber
\end{align}%
\noindent where $\operatorname{Re}p>0$. The convolution theorem then can be specified, see 
\cite{pbm492}, as%
\begin{align}
\overline{f}_{1}\left( p\right) \overline{f}_{2}\left( p\right)
=\int_{\alpha _{1}+\alpha _{2}}^{\beta _{1}+\beta _{2}}\exp \left(
-pt\right) f_{1}\left( t\right) \ast f_{2}\left( t\right) dt  \label{ltconvol}
\end{align}%
\noindent where $f_{1}\left( t\right) \ast f_{2}\left( t\right) $ is the
convolution of $f_{1}\left( t\right) $ and $f_{2}\left( t\right) $ that is
to be obtained from%
\begin{align}
f_{1}\left( t\right) \ast f_{2}\left( t\right) =\int_{\max \left( \alpha
_{1};\text{ }t-\beta _{2}\right) }^{\min \left( \beta _{1};\text{ }t-\alpha
_{2}\right) }f_{1}\left( \tau \right) f_{2}\left( t-\tau \right) d\tau 
\label{iltconvol}
\end{align}%
\medskip
\section{Inverse Laplace transforms for products of parabolic cylinder functions}
\noindent This Section derives several inverse Laplace transforms for products of parabolic cylinder functions in which the sign of the arguments may differ and utilizes these results to obtain new integral representations for products of (complementary) error functions.
\bigskip
\begin{theorem}
Let $\nu$ and $\mu$ be two complex numbers with $\operatorname{Re}\nu <1$ and $\operatorname{Re}\mu <\min \left[ 1-\operatorname{Re}\nu ,2+\operatorname{Re}
\nu \right] $. Then, the following inverse Laplace transform holds for $\operatorname{Re}p>0$, $x>0$, $\left\vert \arg y\right\vert <\pi$, $y>0$ 
\begin{align}
&p^{-1/2}\exp \left( \tfrac{1}{2}p\left( y-x\right) \right) D_{\mu }\left(
2^{1/2}y^{1/2}p^{1/2}\right) \{D_{\nu }\left( -2^{1/2}x^{1/2}p^{1/2}\right) 
\nonumber \\
&\hspace{2cm}-D_{\nu }\left( 2^{1/2}x^{1/2}p^{1/2}\right) \}= \label{ilt1}\\
&\dfrac{2^{\left( \mu -\nu \right) /2}\sqrt{\pi }}{\Gamma \left[ 1+\left(
\nu -\mu \right) /2\right] \Gamma \left[ -\nu \right] }\int_{0}^{x}\exp
\left( -pt\right) t^{\left( \nu -\mu \right) /2}\left( x-t\right) ^{-\left(1+\nu\right) /2} \nonumber \\
&\hspace{1cm}\times \left( y+t\right) ^{\mu /2} \mathstrut_2 F_1\left( -\dfrac{\mu }{2},%
\dfrac{1+\nu }{2};1+\dfrac{\nu -\mu }{2};\dfrac{t\left( x-y-t\right) }{%
\left( x-t\right) \left( y+t\right) }\right) dt \nonumber \\
&+\dfrac{2^{2+\left( \mu +\nu \right) /2}\sqrt{\pi }y^{1/2}x^{1/2}}{\Gamma %
\left[ -\mu /2\right] \Gamma \left[ -\nu /2\right] }\int_{x}^{\infty }\exp
\left( -pt\right) t^{\left( \nu -1\right) /2}\left( t-x\right) ^{-\left(
1+\mu +\nu \right) /2} \nonumber \\
&\hspace{1cm}\times \left( y-x+t\right) ^{\left( \mu -1\right) /2} \mathstrut_2 F_1\left( 
\dfrac{1-\mu }{2},\dfrac{1-\nu }{2};\dfrac{3}{2};\dfrac{xy}{t\left(
y-x+t\right) }\right) dt \nonumber
\end{align}
\end{theorem}
\noindent where $\mathstrut_2 F_1(a,b;c;z)$ denotes the Gaussian hypergeometric function, see \cite{as72}. \\
\begin{proof}
The inverse Laplace transform in Equation (5) on p. 290 in \cite{emoti154} is
\begin{align}
& \Gamma\left[  \nu\right]  \exp\left(  \tfrac{1}{2}ap\right)  D_{-2\nu
}\left(  2^{1/2}a^{1/2}p^{1/2}\right)  = \label{iltpcf1} \\
&\hspace{1cm} \int_{0}^{\infty}\exp\left(
-pt\right)  2^{-\nu}a^{1/2}t^{\nu-1}\left(  t+a\right)  ^{-\nu-1/2}dt \nonumber \\
& \hspace{2cm}\left[  \operatorname{Re}p>0,\operatorname{Re}\nu>0,\left\vert
\arg a\right\vert <\pi\right]\nonumber 
\end{align}
\noindent and the inverse Laplace transform in Equation (3.33.2.2) in \cite{pbm592} is
\begin{align}
& \exp \left( -xp\right) \Phi \left( a;b;xp\right) = \label{iltkum1} \\
& \hspace{1cm}\dfrac{x^{1-b}\Gamma
\left[ b\right] }{\Gamma \left[ b-a\right] \Gamma \left[ a\right] }
\int_{0}^{x}\exp \left( -pt\right) t^{b-a-1}\left( x-t\right) ^{a-1}dt \nonumber \\
& \hspace{2cm}\left[ \operatorname{Re}p>0,\operatorname{Re}b>\operatorname{Re}a>0,x>0\right] \nonumber   
\end{align}
\noindent These two inverse Laplace transforms, in the notation of Theorem 3.1, are rewritten as
\begin{align}
& \Gamma\left[  -\mu/2\right]  \exp\left(  \tfrac{1}{2}yp\right)  D_{\mu
}\left(  2^{1/2}y^{1/2}p^{1/2}\right)  =\nonumber \\
& \hspace{1cm}\int_{0}^{\infty}\exp\left(
-pt\right)  2^{\mu/2}y^{1/2}t^{-\mu/2-1}\left(  t+y\right)  ^{\left(
\mu-1\right)  /2}dt\nonumber\\
& \hspace{2cm}\left[  \operatorname{Re}p>0,\operatorname{Re}\mu<0,\left\vert
\arg y\right\vert <\pi\right] \label{iltpcf2}%
\end{align}
\noindent and
\begin{align}
& x^{1/2} \frac{2}{\sqrt{\pi}} \Gamma \left[ 1+\nu /2\right] \Gamma \left[ \left( 1-\nu
\right) /2\right] \exp \left( -xp\right) \Phi \left( \dfrac{%
1-\nu }{2};\dfrac{3}{2};px\right) =  \nonumber \\
& \hspace{1cm}\int_{0}^{x}\exp \left( -pt\right) t^{\nu
/2}\left(x-t\right) ^{-\left( 1+\nu \right) /2}dt\nonumber \\
& \hspace{2cm}\left[ \operatorname{Re}p>0,-2<\operatorname{Re}\nu <1,x>0\right]
 \label{iltkum2}
\end{align}
\noindent The original functions $f_{1}\left(  t\right)  $ and $f_{2}\left(
t\right) $ are taken from the inverse Laplace transforms (\ref{iltpcf2}) and (\ref{iltkum2}), respectively, with%
\[
f_{1}\left(  t\right)  =2^{\mu/2}y^{1/2}t^{-\mu/2-1}\left(  t+y\right)
^{\left(  \mu-1\right)  /2}\text{ and }f_{2}\left(  t\right)  =t^{\nu
/2}\left( x-t\right)  ^{-\left(  1+\nu\right)  /2}
\]
\noindent The integration limits in (\ref{ltconvol}) and (\ref{iltconvol}) are $\beta_{1}=\infty,\beta_{2}=x$ and $\alpha_{1}=\alpha_{2}=0$ such that the convolution integral is given by%
\begin{equation}%
\begin{tabular}
[c]{lll}%
$f_{1}\left(  t\right)  \ast f_{2}\left(  t\right)  $ & $=%
\displaystyle\int\nolimits_{0}^{t}
f_{1}\left(  \tau\right)  f_{2}\left(  t-\tau\right)  d\tau$ & $\hspace
{1cm}t<x$\\
&  & \\
& $=%
\displaystyle\int\nolimits_{t-x}^{t}
f_{1}\left(  \tau\right)  f_{2}\left(  t-\tau\right)  d\tau$ & $\hspace
{1cm}t>x$%
\end{tabular}
\label{convol1}%
\end{equation}
\noindent First, the convolution integral for $t<x$ is%
\begin{align}
& f_{1}\left(  t\right)  \ast f_{2}\left(  t\right)  =\int_{0}^{t}2^{\mu
/2}y^{1/2}\tau^{-\mu/2-1}\left(  \tau+y\right)  ^{\left(  \mu-1\right)
/2}\left(  t-\tau\right)  ^{\nu/2}\nonumber  \\
& \hspace{1cm}\times \left( x-\left(
t-\tau\right)  \right)  ^{-\left(  1+\nu\right)  /2}d\tau\nonumber 
\end{align}
\noindent The substitution $\tau =tu$ allows to rewrite the integral as%
\begin{align}
& f_{1}\left( t\right) \ast f_{2}\left( t\right) =2^{\mu /2}t^{\left( \nu
-\mu \right) /2}y^{\mu /2}\left( x-t\right) ^{-\left( 1+\nu \right) /2} \nonumber \\
& \hspace{1cm}\times \int_{0}^{1}u^{-\mu /2-1}\left( 1+\dfrac{t}{y}u\right)
^{\left( \mu -1\right) /2}\left( 1-u\right) ^{\nu /2}\left( 1-\dfrac{t}{t-x}%
u\right) ^{-\left( 1+\nu \right) /2}du\nonumber 
\end{align}%
\noindent The integral in the latter equation can be expressed in terms of the Appell
function $F_{1}\left( a,b_{1},b_{2};c;z_{1},z_{2}\right) $ given that%
\begin{align}
& \dfrac{\Gamma \left[ a\right] \Gamma \left[ c-a\right] }{\Gamma \left[ c%
\right] }F_{1}\left( a,b_{1},b_{2};c;z_{1},z_{2}\right)= \nonumber \\
& \hspace{1cm} \int_{0}^{1}u^{a-1}\left( 1-u\right) ^{c-a-1}\left( 1-z_{1}u\right)
^{-b_{1}}\left( 1-z_{2}u\right) ^{-b_{2}}du\nonumber 
\end{align}%
\noindent for $\operatorname{Re}c>\operatorname{Re}a>0$, see Equation (5) on p. 231 in \cite{emoth153}. This
gives%
\begin{align}
& f_{1}\left( t\right) \ast f_{2}\left( t\right) =2^{\mu /2}t^{\left( \nu
-\mu \right) /2}y^{\mu /2}\left( x-t\right) ^{-\left( 1+\nu \right) /2}%
\dfrac{\Gamma \left[ -\mu /2\right] \Gamma \left[ 1+\left( \nu /2\right) %
\right] }{\Gamma \left[ 1+\left( \nu -\mu \right) /2\right] } \nonumber \\
& \hspace{1cm}\times F_{1}\left( -\dfrac{\mu }{2},\dfrac{1+\nu }{2},\dfrac{1-\mu }{2};1+\dfrac{\nu -\mu }{2};\dfrac{t}{t-x},-\dfrac{t}{y}%
\right)\nonumber 
\end{align}
\noindent The above Appell function can further be simplified into the Gaussian hypergeometric function given
\[
F_{1}\left( a,b_{1},b_{2};b_{1}+b_{2};z_{1},z_{2}\right) =\left(
1-z_{2}\right) ^{-a} \mathstrut_2 F_1\left( a,b_{1};b_{1}+b_{2};\dfrac{z_{1}-z_{2}}{1-z_{2}}%
\right)
\]%
\noindent see Equation (1) on p. 238 in \cite{emoth153}. The final expression for the convolution integral for $t<x$ then is%
\begin{align}
& f_{1}\left( t\right) \ast f_{2}\left( t\right) =2^{\mu /2}t^{\left( \nu
-\mu \right) /2}\left( x-t\right) ^{-\left( 1+\nu \right) /2}\left(
y+t\right) ^{\mu /2}\dfrac{\Gamma \left[ -\mu /2\right] \Gamma \left[
1+\left( \nu /2\right) \right] }{\Gamma \left[ 1+\left( \nu -\mu \right) /2%
\right] }  \nonumber \\
& \hspace{1cm}\times \mathstrut_2 F_1\left( -\dfrac{\mu }{2},\dfrac{1+\nu }{2};1+\dfrac{\nu
-\mu }{2};\dfrac{t\left( t+y-x\right) }{\left( t-x\right)\left( y+t\right)  }%
\right) \hspace{1cm}t<x  \label{convol2}
\end{align}%
\noindent Second, the convolution integral for $t>x$ is given by
\begin{align}
& f_{1}\left(  t\right)  \ast f_{2}\left(  t\right)  =\int_{t-x}^{t}2^{\mu
/2}y^{1/2}\tau^{-\mu/2-1}\left(  \tau+y\right)  ^{\left(  \mu-1\right)
/2}\left(  t-\tau\right)  ^{\nu/2} \nonumber \\
& \hspace{1cm} \times \left( x-\left(
t-\tau\right)  \right)  ^{-\left(  1+\nu\right)  /2}d\tau\nonumber 
\end{align}
\noindent The treatment of this convolution integral is similar to that of the integral for $t<x$ such that only the main steps are mentioned. The substitutions $\tau =s-x+t$ and $s=xu$ express the integral in terms of the Appell function $F_{1}$ that again can be simplified into the Gaussian hypergeometric function. The convolution integral for $t>x$ then is given by%
\begin{align}
& f_{1}\left( t\right) \ast f_{2}\left( t\right) =\dfrac{1}{\sqrt{\pi }}%
x^{1/2}y^{1/2}2^{1+\left( \mu /2\right) }\left( t-x\right) ^{-\left( 1+\mu
+\nu \right) /2}\left( y+t-x\right) ^{\left( \mu -1\right) /2}t^{\left( \nu
-1\right) /2}  \nonumber \\
& \hspace{1cm}\times \Gamma \left[ \left( 1-\nu \right) /2\right] \Gamma %
\left[ 1+\left( \nu /2\right) \right] \mathstrut_2 F_1\left( \dfrac{1-\mu }{2},\dfrac{1-\nu 
}{2};\dfrac{3}{2};\dfrac{xy}{t\left( t+y-x\right) }\right) \hspace{0.4cm}t>x 
 \label{convol3}
\end{align}
\noindent of which the derivation also used the following linear transformation formula%
\begin{align}
\mathstrut_2 F_1\left( a,b;c;z\right) =\left( 1-z\right) ^{-a}\mathstrut_2 F_1\left( a,c-b;c;\dfrac{z}{z-1}%
\right)\nonumber 
\end{align}
\noindent see Equation (15.3.4) in \cite{as72}.

Plugging (\ref{convol2}) and (\ref{convol3}) into the convolution integral (\ref{convol1}) then gives
\begin{align}
& \exp \left( \tfrac{1}{2}py-px\right) D_{\mu }\left(
2^{1/2}y^{1/2}p^{1/2}\right) \Phi \left( \dfrac{1-\nu }{2};\dfrac{3}{2};px\right)
   \label{ilt2}\\
&\dfrac{2^{\left( \mu /2\right) -1}\sqrt{\pi }x^{-1/2}}{\Gamma \left[
1+\left( \nu -\mu \right) /2\right] \Gamma \left[ \left( 1-\nu \right) /2%
\right] }\int_{0}^{x}\exp \left( -pt\right) t^{\left( \nu -\mu \right)
/2}\left( x-t\right) ^{-\left( 1+\nu \right) /2} \nonumber \\
&\hspace{1cm}\times \left( y+t\right) ^{\mu /2}\mathstrut_2 F_1\left( -\dfrac{\mu }{2},%
\dfrac{1+\nu }{2};1+\dfrac{\nu -\mu }{2};\dfrac{t\left( x-y-t\right) }{%
\left( x-t\right) \left( y+t\right) }\right) dt \nonumber \\
&+\dfrac{2^{\mu /2}y^{1/2}}{\Gamma \left[ -\mu /2\right] }\int_{x}^{\infty
}\exp \left( -pt\right) t^{\left( \nu -1\right) /2}\left( t-x\right)
^{-\left( 1+\mu +\nu \right) /2} \nonumber \\
&\hspace{1cm}\times \left( y-x+t\right) ^{\left( \mu -1\right) /2}\mathstrut_2 F_1\left( 
\dfrac{1-\mu }{2},\dfrac{1-\nu }{2};\dfrac{3}{2};\dfrac{xy}{t\left(
y-x+t\right) }\right) dt\nonumber 
\end{align}
\noindent in which the recurrence and duplication formulas of the gamma function were employed to simplify expressions given that
\begin{align}
& \Gamma \left[ 1+z\right]  = z\Gamma \left[ z\right] \nonumber \\
& \Gamma \left[ 2z\right]  = \dfrac{1}{\sqrt{2\pi }}2^{2z-\tfrac{1}{2}}\Gamma\left[ z\right] \Gamma \left[ z+\tfrac{1}{2}\right]\nonumber 
\end{align}
\noindent see Equations (6.1.15) and (6.1.18) in \cite{as72}.

Finally, plugging the definition (\ref{sum1}) into (\ref{ilt2}) and simplifying gives the inverse Laplace transform (\ref{ilt1}).
\end{proof}

\noindent The parabolic cylinder function specializes into the complementary error function when its order is at $-1$. The inverse Laplace transform (\ref{ilt1}) thus can be used to obtain an integral representation for the product of complementary error functions. However, this result will not be shown here as its integrand contains an inverse trigonometric function rather than the rational functions that are typical for existing integral representations, see for instance \cite{ng69,gr14}. Instead, the term $p^{-1/2}$ in inverse Laplace transforms such as (\ref{ilt1}) will be removed given that the resulting relations yield integrands in which such rational functions emerge. This will be illustrated in Theorem 3.2 and its Corollary 3.2.1.

\begin{theorem}
Let $\nu$ and $\mu$ be two complex numbers with $\operatorname{Re}\nu <1$ and $\operatorname{Re}\mu <\min \left[ 1-\operatorname{Re}\nu ,2+\operatorname{Re}
\nu \right] $.
 Then, the following inverse Laplace transform holds for $\operatorname{Re}p>0$, $x>0$, $\left\vert \arg y\right\vert <\pi$, $y>0$
\begin{align}
&\exp \left( \tfrac{1}{2}p\left( y-x\right) \right) D_{\mu }\left(
2^{1/2}y^{1/2}p^{1/2}\right) \left\{ D_{\nu }\left(
-2^{1/2}x^{1/2}p^{1/2}\right) \right.  \nonumber \\
&\hspace{2cm}\left. -D_{\nu }\left( 2^{1/2}x^{1/2}p^{1/2}\right) \right\} \label{ilt3}
\\
&\dfrac{2^{\left( \mu -\nu \right) /2}\sqrt{\pi }y^{-1/2}}{\Gamma \left[
\left( 1-\mu +\nu \right) /2\right] \Gamma \left[ -\nu \right] }%
\int_{0}^{x}\exp \left( -pt\right) t^{-\left( 1+\mu -\nu \right) /2}\left(
x-t\right) ^{-\left( 1+\nu \right) /2} \nonumber \\
&\hspace{0.5cm}\times \left( y+t\right) ^{\left( 1+\mu \right) /2}\left\{\mathstrut_2 F_1\left( -\dfrac{1+\mu }{2},\dfrac{1+\nu }{2};\dfrac{1-\mu +\nu }{2};\dfrac{%
t\left( x-y-t\right) }{\left( x-t\right) \left( y+t\right) }\right) \right. 
\nonumber \\
&\hspace{0.5cm}+\left. \dfrac{\mu t}{\left(1-\mu +\nu \right)\left( y+t\right)}\mathstrut_2 F_1\left( 
\dfrac{1-\mu }{2},\dfrac{1+\nu }{2};\dfrac{3-\mu +\nu }{2};\dfrac{t\left(
x-y-t\right) }{\left( x-t\right) \left( y+t\right) }\right) \right\} dt \nonumber \\
&+\dfrac{2^{\left( 4+\mu +\nu \right) /2}\sqrt{\pi }x^{1/2}}{\Gamma \left[
-\left( 1+\mu \right) /2\right] \Gamma \left[ -\nu /2\right] }%
\int_{x}^{\infty }\exp \left( -pt\right) t^{\left( \nu -1\right) /2}\left(
t-x\right) ^{-\left( 2+\mu +\nu \right) /2} \nonumber \\
&\hspace{0.5cm}\times \left( y-x+t\right) ^{\mu /2}\left\{ \mathstrut_2 F_1\left( -\dfrac{%
\mu }{2},\dfrac{1-\nu }{2};\dfrac{3}{2};\dfrac{xy}{t\left( y-x+t\right) }%
\right) \right. \nonumber \\
&\hspace{0.5cm}\left. -\dfrac{\mu\left(t-x\right) }{\left(1+\mu \right)\left(y-x+t\right)}\mathstrut_2 F_1\left( \dfrac{%
2-\mu }{2},\dfrac{1-\nu }{2};\dfrac{3}{2};\dfrac{xy}{t\left( y-x+t\right) }%
\right) \right\} dt \nonumber
\end{align}
\end{theorem}
\begin{proof}
The recurrence relation of the parabolic cylinder function is given by
\[
zD_{\mu }\left( z\right) =D_{\mu +1}\left( z\right) +\mu D_{\mu -1}\left(
z\right)
\]%
\noindent see Equation (14) on p. 119 in \cite{emoth253}.
Replacing $z$ by $2^{1/2}y^{1/2}p^{1/2}$ and multiplying by $p^{-1/2}\exp \left( \tfrac{1}{2}p\left( y-x\right)
\right) \left\{ D_{\nu }\left( -2^{1/2}x^{1/2}p^{1/2}\right) - D_{\nu }\left(
2^{1/2}x^{1/2}p^{1/2}\right) \right\} $
gives%
\begin{align}
&2^{1/2}y^{1/2}\exp \left( \tfrac{1}{2}p\left( y-x\right) \right) D_{\mu
}\left( 2^{1/2}y^{1/2}p^{1/2}\right) \left\{ D_{\nu }\left(
-2^{1/2}x^{1/2}p^{1/2}\right) \right. \nonumber \\
&\hspace{2cm}\left. -D_{\nu }\left( 2^{1/2}x^{1/2}p^{1/2}\right) \right\} = \label{help1}
\\
&\hspace{1cm}p^{-1/2}\exp \left( \tfrac{1}{2}p\left( y-x\right) \right) D_{\mu
+1}\left( 2^{1/2}y^{1/2}p^{1/2}\right) \left\{ D_{\nu }\left(
-2^{1/2}x^{1/2}p^{1/2}\right) \right.  \nonumber \\
&\hspace{2cm}\left. -D_{\nu }\left( 2^{1/2}x^{1/2}p^{1/2}\right) \right\} \nonumber
\\
&\hspace{1cm}+\mu p^{-1/2}\exp \left( \tfrac{1}{2}p\left( y-x\right) \right) D_{\mu
-1}\left( 2^{1/2}y^{1/2}p^{1/2}\right) \left\{ D_{\nu }\left(
-2^{1/2}x^{1/2}p^{1/2}\right) \right. \nonumber \\
&\hspace{2cm}\left. -D_{\nu }\left( 2^{1/2}x^{1/2}p^{1/2}\right) \right\} \nonumber
\end{align}%
\noindent Plugging the transform (\ref{ilt1}) into (\ref{help1}) and simplifying gives (\ref{ilt3}).
\end{proof}
\begin{corollary}
The relation between the parabolic cylinder function and the complementary error function is given by
\[
D_{-1}\left( z\right) =\sqrt{\dfrac{\pi }{2}}\exp \left( \dfrac{z^{2}}{4}%
\right) \operatorname{erfc}\left( \dfrac{z}{\sqrt{2}}\right)
\]%
see Equation (9.254.1) in \cite{gr14} in which erfc$(z)$ denotes the complementary error function. Equations (E.3c) and (E.3d) in \cite{bchl92} specify the following relations between the error function and its complement%
\begin{align*}
&\operatorname{erfc}\left( z\right) +\operatorname{erf}\left( z\right) =1 \\
&\operatorname{erfc}\left( -z\right) =1+\operatorname{erf}\left( z\right)
\end{align*}and thus
\begin{align}
\operatorname{erfc}\left( -z\right) -\operatorname{erfc}\left( z\right) =2\operatorname{erf}\left(
z\right) \label{help2}
\end{align}
\noindent where erf$(z)$ denotes the error function. The below derivations also use the following properties of the Gaussian hypergeometric
function%
\begin{align*}
&\mathstrut_2 F_1\left( 0,b;c;z\right) =\mathstrut_2 F_1\left( a,0;c;z\right) =1 \\
&\mathstrut_2 F_1\left( 1,\dfrac{3}{2};\dfrac{3}{2};z\right) =\dfrac{1}{1-z}
\end{align*}%
see Equations (15.1.1) and (15.1.8) in \cite{as72}.
Plugging the transform (\ref{ilt1}) into (\ref{help1}), using $\mu =\nu =-1$ and (\ref{help2}) gives the following inverse Laplace transform for the product of two (complementary) error functions
\begin{align}
&\exp \left( py\right) \ \operatorname{erfc}\left( y^{1/2}p^{1/2}\right) \ \operatorname{erf
}\left( x^{1/2}p^{1/2}\right) = \label{ilterf1}  \\
&\hspace{1cm}\dfrac{1}{\pi }\int_{0}^{x}\exp \left( -pt\right) \dfrac{\sqrt{y}}{\sqrt{t}
\left( y+t\right) }dt-\dfrac{1}{\pi }\int_{x}^{\infty }\exp \left(
-pt\right) \dfrac{\sqrt{x}}{\sqrt{y-x+t}\left( y+t\right) }dt \nonumber \\
&\hspace{2cm}[\operatorname{Re}p>0, \left\vert \arg y\right\vert <\pi, y>0, \left\vert \arg x\right\vert <\pi, x\geqslant 0]\nonumber
\end{align}
Using $p=1$ and setting $a$ and $b$ at $y^{1/2}$ and $x^{1/2}$, respectively, then gives the following integral
representation
\begin{align}
&\operatorname{erfc}\left( a\right) \ \operatorname{erf}\left( b\right) = \label{intrep1} \\
&\dfrac{a\exp \left( -a^{2}\right) }{\pi }\int_{0}^{b^{2}}\dfrac{\exp
\left( -t\right) }{\left( t+a^{2}\right) \sqrt{t}}dt-\dfrac{b\exp \left(
-\left( a^{2}+b^{2}\right) \right) }{\pi }\int_{0}^{\infty }\dfrac{\exp
\left( -t\right) }{\left( t+a^{2}+b^{2}\right) \sqrt{t+a^{2}}}dt \nonumber \\
&\hspace{2cm}[\operatorname{Re}a>0, a>0,\operatorname{Re}b\geqslant 0, b\geqslant 0] \nonumber
\end{align}

\noindent which is not present in, for instance, the extensive overview in \cite{ng69}.
\end{corollary}
\medskip

\begin{theorem}
Let $\nu$ and $\mu$ be two complex numbers with $\operatorname{Re}\nu <1$ and $\operatorname{Re}\mu <\min \left[ 1-\operatorname{Re}\nu ,2+\operatorname{Re}
\nu \right] $.
 Then, the following inverse Laplace transform holds for $\operatorname{Re}p>0$, $\left\vert \arg x\right\vert <\pi$, $x\geqslant 0$, $\left\vert \arg y\right\vert <\pi$, $y>0$
\begin{align}
& p^{-1/2}\exp \left( \tfrac{1}{2}p\left( y-x\right) \right) D_{\mu }\left(
2^{1/2}y^{1/2}p^{1/2}\right) \{D_{\nu }\left( -2^{1/2}x^{1/2}p^{1/2}\right) \nonumber 
\\
& \hspace{2cm}+D_{\nu }\left( 2^{1/2}x^{1/2}p^{1/2}\right) \}= \label{ilt4} \\
& \dfrac{2^{\left( \mu -\nu \right) /2}\sqrt{\pi }}{\Gamma %
\left[ 1+\left( \nu -\mu \right) /2\right] \Gamma \left[ -\nu \right] }%
\int_{0}^{x}\exp \left( -pt\right) t^{\left( \nu -\mu \right) /2}\left(
x-t\right) ^{-\left(1+ \nu\right)/2 } \nonumber \\
& \hspace{1cm}\times \left( y+t\right) ^{\mu /2}\mathstrut_2 F_1\left( 
-\dfrac{\mu }{2},\dfrac{1+\nu }{2};1+\dfrac{\nu -\mu }{2};\dfrac{t\left(
x-y-t\right) }{\left( x-t\right) \left( y+t\right) }\right) dt \nonumber \\
& +\dfrac{2^{1+\left( \mu +\nu \right) /2}\sqrt{\pi }}{\Gamma \left[ \left(
1-\nu \right) /2\right] \Gamma \left[ \left( 1-\mu \right) /2\right] }%
\int_{x}^{\infty }\exp \left( -pt\right) t^{\nu /2}\left( t-x\right)
^{-\left( 1+\mu +\nu \right) /2} \nonumber \\
& \hspace{1cm}\times \left( y-x+t\right) ^{\mu /2}\mathstrut_2 F_1\left( -\dfrac{\mu }{2},-%
\dfrac{\nu }{2};\dfrac{1}{2};\dfrac{xy}{t\left( y-x+t\right) }\right) dt\nonumber 
\end{align}
\end{theorem}
\begin{proof}
The inverse Laplace transform in Equation (6) on p. 290 in \cite{emoti154} is%
\begin{align*}
& \Gamma\left[  \nu\right]  p^{-1/2} \exp\left(  \tfrac{1}{2}ap\right)  D_{1-2\nu
}\left(  2^{1/2}a^{1/2}p^{1/2}\right)  = \\
& \hspace{1cm}\int_{0}^{\infty}\exp\left(
-pt\right)  2^{1/2-\nu}t^{\nu-1}\left(  t+a\right)  ^{1/2-\nu}dt\nonumber \\
& \hspace{2cm}\left[  \operatorname{Re}p>0,\operatorname{Re}\nu>0,\left\vert
\arg a\right\vert <\pi\right]\nonumber 
\end{align*}
\noindent which in the notation of Theorem 3.3 gives
\begin{align}
& \Gamma\left[  \left(  1-\mu\right)  /2\right]  p^{-1/2}\exp\left(  \tfrac
{1}{2}yp\right)  D_{\mu}\left(  2^{1/2}y^{1/2}p^{1/2}\right)  = \nonumber \\
&\hspace{1cm}\int
_{0}^{\infty}\exp\left(  -pt\right)  2^{\mu/2}t^{-\left(  \mu+1\right)
/2}\left(  t+y\right)  ^{\mu/2}dt\nonumber \\
& \hspace{2cm}\left[  \operatorname{Re}p>0,\operatorname{Re}\mu<1,\left\vert
\arg y\right\vert <\pi\right]\label{iltpcf3}%
\end{align}
\noindent The inverse Laplace transform (\ref{iltkum1}) is specialized for $a=-\tfrac{\nu }{2}$ and $b=\tfrac{1}{2}$ and gives
\begin{align}
& \dfrac{x^{-1/2}}{\sqrt{\pi }}\Gamma \left[ \left( 1+\nu \right) /2\right]
\Gamma \left[ -\nu /2\right] \exp \left( -xp\right) \Phi \left( -\dfrac{\nu 
}{2};\dfrac{1}{2};xp\right) = \nonumber \\
& \hspace{1cm}\int_{0}^{x}\exp \left( -pt\right) t^{\left( \nu -1\right)
/2}\left( x-t\right) ^{-\left( \nu /2\right) -1}dt \nonumber \\
& \hspace{2cm}\left[ \operatorname{Re}p>0,-1<\operatorname{Re}\nu <0,x>0\right]
\label{iltkum3}
\end{align}
\noindent The original functions $f_{1}\left(  t\right)  $ and $f_{2}\left(
t\right) $ are taken from the inverse Laplace transforms (\ref{iltpcf3}) and (\ref{iltkum3}), respectively%
\[
f_{1}\left(  t\right)  =2^{\mu/2}t^{-\left(  \mu+1\right)
/2}\left(  t+y\right)  ^{\mu/2}\text{ and }f_{2}\left(  t\right)  =t^{\left(\nu -1\right)
/2}\left( x-t\right)  ^{-\left(\nu/2\right) -1}
\]
\noindent Using steps akin to those used in the proof of Theorem 3.1 then yields
\begin{align}
& p^{-1/2}\exp \left( \tfrac{1}{2}py-px\right) D_{\mu }\left(
2^{1/2}y^{1/2}p^{1/2}\right) \Phi \left( -\dfrac{\nu }{2};\dfrac{1}{2}%
;px\right) = \label{ilt5} \\
& \dfrac{2^{\mu /2}\sqrt{\pi }x^{1/2}y^{1/2}}{\Gamma \left[ 1+\left( \nu
-\mu \right) /2\right] \Gamma \left[ -\nu /2\right] }\int_{0}^{x}\exp \left(
-pt\right) t^{\left( \nu -\mu \right) /2}\left( x-t\right) ^{-1-\left( \nu
/2\right) } \nonumber \\
& \hspace{1cm}\times \left( y+t\right) ^{\left( \mu -1\right) /2}\mathstrut_2 F_1\left( 
\dfrac{1-\mu }{2},1+\dfrac{\nu }{2};1+\dfrac{\nu -\mu }{2};\dfrac{t\left(
x-y-t\right) }{\left( x-t\right) \left( y+t\right) }\right) dt \nonumber \\
& +\dfrac{2^{\mu /2}}{\Gamma \left[ \left( 1-\mu \right) /2\right] }%
\int_{x}^{\infty }\exp \left( -pt\right) t^{\nu /2}\left( t-x\right)
^{-\left( 1+\mu +\nu \right) /2} \nonumber \\
& \hspace{1cm}\times \left( y-x+t\right) ^{\mu /2}\mathstrut_2 F_1\left( -\dfrac{\mu }{2},-%
\dfrac{\nu }{2};\dfrac{1}{2};\dfrac{xy}{t\left( y-x+t\right) }\right) dt\nonumber 
\end{align}
\noindent The first integral in (\ref{ilt5}) can be rewritten via the following linear transformation formula for the Gaussian
hypergeometric function%
\begin{align}
\mathstrut_2 F_1\left( a,b;c;z\right) =\left( 1-z\right) ^{c-a-b}\mathstrut_2 F_1\left( c-a,c-b;c;z\right)  \label{15.3.3}
\end{align}
\noindent see Equation (15.3.3) in \cite{as72}. Combining the resulting expression for the transform (\ref{ilt5}) with the definition (\ref{sum2}) then gives the inverse Laplace transform (\ref{ilt4}).
\end{proof}
\medskip
\noindent Theorem 3.4 specifies the inverse Laplace transform for the product of two parabolic cylinder functions of which the arguments have opposite sign and Corollary 3.4.1 specializes this expression for a single parabolic cylinder function with negative sign in the argument.
\begin{theorem}
Let $\nu$ and $\mu$ be two complex numbers with $\operatorname{Re}\nu <1$ and $\operatorname{Re}\mu <\min \left[ 1-\operatorname{Re}\nu ,2+\operatorname{Re}
\nu \right]$.
 Then, the following inverse Laplace transform holds for $\operatorname{Re}p>0$, $x>0$, $\left\vert \arg y\right\vert <\pi$, $y>0$
\begin{align}
&p^{-1/2}\exp \left( \tfrac{1}{2}p\left( y-x\right) \right) D_{\mu }\left(
2^{1/2}y^{1/2}p^{1/2}\right) D_{\nu }\left( -2^{1/2}x^{1/2}p^{1/2}\right) =  \label{ilt6}
\\
&\dfrac{2^{\left( \mu -\nu \right) /2}\sqrt{\pi }}{\Gamma \left[ 1+\left(
\nu -\mu \right) /2\right] \Gamma \left[ -\nu \right] }\int_{0}^{x}\exp
\left( -pt\right) t^{\left( \nu -\mu \right) /2}\left( x-t\right) ^{-\left(
1+\nu \right) /2} \nonumber \\
&\times \left( y+t\right) ^{\mu /2}\mathstrut_2 F_1\left( -\dfrac{\mu }{2},%
\dfrac{1+\nu }{2};1+\dfrac{\nu -\mu }{2};\dfrac{t\left( x-y-t\right) }{%
\left( x-t\right) \left( y+t\right) }\right) dt \nonumber \\
&+\dfrac{2^{1+\left( \mu +\nu \right) /2}\sqrt{\pi }x^{1/2}y^{1/2}}{\Gamma %
\left[ -\mu /2\right] \Gamma \left[ -\nu /2\right] }\int_{x}^{\infty }\exp
\left( -pt\right) t^{\left( \nu -1\right) /2}\left( t-x\right) ^{-\left(
1+\mu +\nu \right) /2} \nonumber \\
&\times \left( y-x+t\right) ^{\left( \mu -1\right) /2}\left\{
\mathstrut_2 F_1\left( \dfrac{1-\mu }{2},\dfrac{1-\nu }{2};\dfrac{3}{2};\dfrac{xy}{t\left(
y-x+t\right) }\right) \right. \nonumber  \\
&\left. +\dfrac{\Gamma \left[ -\mu /2\right] \Gamma \left[ -\nu
/2\right] }{\Gamma \left[ \left( 1-\mu \right) /2\right] \Gamma \left[
\left( 1-\nu \right) /2\right] }\left( \dfrac{t\left( y-x+t\right) }{4xy}%
\right) ^{1/2}\mathstrut_2 F_1\left( -\dfrac{\mu }{2},-\dfrac{\nu }{2};\dfrac{1}{2};\dfrac{%
xy}{t\left( y-x+t\right) }\right) \right\}dt \nonumber 
\end{align}
\end{theorem}
\begin{proof}
\noindent The transform (\ref{ilt6}) is obtained by adding the inverse Laplace transforms (\ref{ilt1}) and (\ref{ilt4}) and simplifying the resulting expression.
\end{proof}

\begin{corollary}
Using $y=0$ and the following two properties
\begin{align*}
&D_{\mu }\left( 0\right) =\dfrac{2^{\mu /2}\sqrt{\pi }}{\Gamma \left[
(1-\mu )/2\right] } \\
&\mathstrut_2 F_1\left( a,b;c;1\right) =\dfrac{\Gamma \left[ c\right] \Gamma \left[ c-a-b%
\right] }{\Gamma \left[ c-a\right] \Gamma \left[ c-b\right] }
\end{align*}%
see Equations (46:7:1) in \cite{oms09} and (15.1.20) in \cite{as72}, gives
\begin{align}
&p^{-1/2}\exp \left( -\tfrac{1}{2}px\right) D_{v}\left(
-2^{1/2}x^{1/2}p^{1/2}\right) = \label{ilt6b}\\
&\hspace{1cm}\frac{2^{-\nu /2}\sqrt{\pi }}{\Gamma \left[ -\nu \right]
\Gamma \left[ 1+\nu /2\right] }\int_{0}^{x}\exp \left( -pt\right) t^{\nu
/2}\left( x-t\right) ^{-\left( 1+\nu \right) /2}dt \nonumber \\
&\hspace{1cm}+\frac{2^{\nu /2}}{\Gamma \left[ \left( 1-\nu \right) /2\right]
}\int_{x}^{\infty }\exp \left( -pt\right) t^{\nu /2}\left( t-x\right)
^{-\left( 1+\nu \right) /2}dt \nonumber \\
&\hspace{2cm}\left[ \operatorname{Re}p>0,\operatorname{Re}\nu <1,x>0\right] \nonumber
\end{align}
\end{corollary}

\begin{theorem}
Let $\nu$ and $\mu$ be two complex numbers with $\operatorname{Re}(\nu + \mu)<1$.
 Then, the following inverse Laplace transform holds for $\operatorname{Re}p>0$, $\left\vert \arg x\right\vert <\pi$, $x \geqslant 0$, $\left\vert \arg y\right\vert <\pi$, $y \geqslant 0$, $\left\vert \arg x + \arg y\right\vert <\pi$  
\begin{align}
& p^{-1/2}\exp\left(  \tfrac{1}{2}p\left(  y+x\right)  \right)  D_{\mu}\left(
2^{1/2}y^{1/2}p^{1/2}\right)  D_{\nu}\left(  2^{1/2}x^{1/2}p^{1/2}\right)
=\label{ilt7} \\
& \dfrac{2^{\left(  \mu+\nu\right)  /2}}{\Gamma\left[  \left(  1-\mu
-\nu\right)  /2\right]  }\int_{0}^{\infty}\exp\left(  -pt\right)  t^{-\left(
1+\mu+\nu\right)  /2}\left(  y+t\right)  ^{\mu/2}\left(  x+t\right)  ^{\nu
/2}\nonumber \\
& \hspace{1cm}\times \mathstrut_2 F_1\left(  -\dfrac{\mu}{2},-\dfrac{\nu}{2};\dfrac{ 1-\mu-\nu}%
{2}  ;\dfrac{t\left(  x+y+t\right)  }{\left(
x+t\right)  \left(  y+t\right)  }\right)  dt\nonumber 
\end{align}
\noindent which is identical to the transform in Equation (2.1) in \cite{v17}.

\end{theorem}
\begin{proof}
Subtracting the inverse Laplace transform (\ref{ilt1}) from (\ref{ilt3}) gives
\begin{align}
&p^{-1/2}\exp \left( \tfrac{1}{2}p\left( y-x\right) \right) D_{\mu }\left(
2^{1/2}y^{1/2}p^{1/2}\right) D_{\nu }\left( 2^{1/2}x^{1/2}p^{1/2}\right) = \nonumber \\
&+\dfrac{2^{\left( \mu +\nu \right) /2}}{\Gamma \left[ \left( 1-\mu -\nu
\right) /2\right] }\int_{x}^{\infty }\exp \left( -pt\right) t^{\nu /2}\left(
t-x\right) ^{-\left( 1+\mu +\nu \right) /2} \nonumber \\
&\times \left( y-x+t\right) ^{\mu /2}\left\{ \dfrac{\sqrt{\pi }\Gamma \left[
\left( 1-\mu -\nu \right) /2\right] }{\Gamma \left[ \left( 1-\mu \right) /2%
\right] \Gamma \left[ \left( 1-\nu \right) /2\right] }\mathstrut_2 F_1\left( -\dfrac{\mu }{2%
},-\dfrac{\nu }{2};\dfrac{1}{2};\dfrac{xy}{t\left( y-x+t\right) }\right)
\right.  \nonumber \\
&\left. -\dfrac{\sqrt{\pi }\Gamma \left[ \left( 1-\mu -\nu \right) /2\right]
}{\Gamma \left[ -\mu /2\right] \Gamma \left[ -\nu /2\right] }\left( \dfrac{%
4xy}{t\left( y-x+t\right) }\right) ^{1/2}\mathstrut_2 F_1\left( \dfrac{1-\mu }{2},\dfrac{%
1-\nu }{2};\dfrac{3}{2};\dfrac{xy}{t\left( y-x+t\right) }\right) \right\} dt \nonumber 
\end{align}
\noindent in which the linear transformation formula (\ref{15.3.3}) was used.
Subsequently, using the linear transformation formula
\begin{align}
&F\left( a,b;c;z\right) =\dfrac{\Gamma \left[ c\right] \Gamma \left[ c-a-b%
\right] }{\Gamma \left[ c-a\right] \Gamma \left[ c-b\right] }\mathstrut_2 F_1\left(
a,b;a+b-c+1;1-z\right) \nonumber  \\
&+\left( 1-z\right) ^{c-a-b}\dfrac{\Gamma \left[ c\right]
\Gamma \left[ a+b-c\right] }{\Gamma \left[ a\right] \Gamma \left[ b\right] }%
\mathstrut_2 F_1\left( c-a,c-b;c-a-b+1;1-z\right) \nonumber 
\end{align}
\noindent in Equation (15.3.6) in \cite{as72} gives
\begin{align}
& p^{-1/2}\exp\left(  \tfrac{1}{2}p\left(  y-x\right)  \right)  D_{\mu}\left(
2^{1/2}y^{1/2}p^{1/2}\right)  D_{\nu}\left(  2^{1/2}x^{1/2}p^{1/2}\right)
= \nonumber  \\
& \dfrac{2^{\left(  \nu+\mu\right)  /2}}{\Gamma\left[  \left(  1-\mu
-\nu\right]  /2\right)  }\int_{x}^{\infty}\exp\left(  -pt\right)  t^{\nu/2}\left( t-x\right)  ^{-\left(1+\mu+\nu\right)/2}\left(y-x+t\right) ^{\mu
/2}\nonumber \\
& \hspace{1cm}\times \mathstrut_2 F_1\left(  -\dfrac{\mu}{2},-\dfrac{\nu}{2};\dfrac{1-\mu-\nu}%
{2} ;\dfrac{\left(t-x\right)\left(y+t\right)}{t\left(y-x+t\right)}\right)  dt\nonumber 
\end{align}
\noindent Multiplying both sides by $\exp\left(  px  \right)$, using the substitution $s=t-x$ and subsequently re-introducing $t$ then gives
(\ref{ilt7}).
\end{proof}
\bigskip
\noindent As noted earlier, removing the term $p^{-1/2}$ from transforms such as (\ref{ilt7}) allows obtaining integral representations for (complementary) error functions in which the integrand contains rational functions. This is illustrated in Theorem 3.6 and Corollary 3.6.1 in which the integral representation for 1 -- erf$(a)^{2}$ in \cite{ng69} is generalized into 1 -- erf$(a)$erf$(b)$.

\begin{theorem}
Let $\nu$ and $\mu$ be two complex numbers with $\operatorname{Re}(\nu + \mu)<1$.
Then, the following inverse Laplace transform holds for $\operatorname{Re}p>0$, $\left\vert \arg x\right\vert <\pi$, $x>0$, $\left\vert \arg y\right\vert <\pi$, $y>0$, $\left\vert \arg x + \arg y\right\vert <\pi$
\begin{align}
&\exp \left( \tfrac{1}{2}p\left( y+x\right) \right) D_{\mu }\left(
2^{1/2}y^{1/2}p^{1/2}\right) D_{\nu }\left( 2^{1/2}x^{1/2}p^{1/2}\right) =\label{ilt8}\\
& \hspace{1cm}\dfrac{2^{\left( \mu +\nu \right) /2}x^{-1/2}}{\Gamma \left[ -\left( \mu
+\nu \right) /2\right] }\int_{0}^{\infty}\exp \left( -pt\right) t^{-1-\left( \nu
+\mu \right) /2}\left( y+t\right) ^{\mu /2} \nonumber \\
& \hspace{1.5cm}\times \left( x+t\right) ^{\left( 1+\nu \right) /2}\left\{ \mathstrut_2 F_1\left( -\dfrac{%
\mu }{2},-\dfrac{1+\nu }{2};-\dfrac{\mu +\nu }{2};\dfrac{t\left(
x+y+t\right) }{\left( x+t\right) \left( y+t\right) }\right) \right. \nonumber  \\
& \hspace{1cm}\left. -\dfrac{\nu t}{\left(\mu +\nu\right) \left(x+t\right) }\mathstrut_2 F_1\left( -\dfrac{\mu }{2},%
\dfrac{1-\nu }{2};1-\dfrac{\mu +\nu }{2};\dfrac{t\left( x+y+t\right) }{%
\left( x+t\right) \left( y+t\right) }\right) \right\} dt\nonumber 
\end{align}
\end{theorem}
\begin{proof}
The inverse Laplace transform (\ref{ilt8}) is obtained via the above recurrence relation of the parabolic cylinder function. Replacing $z$ by $2^{1/2}x^{1/2}p^{1/2}$ in the recurrence relation and multiplying by $p^{-1/2}\exp \left( \tfrac{1}{2}p\left( y+x\right) \right) D_{\mu }\left(2^{1/2}y^{1/2}p^{1/2}\right)$ gives
\begin{align}
&\exp \left( \tfrac{1}{2}p\left( y+x\right) \right) D_{\mu }\left(
2^{1/2}y^{1/2}p^{1/2}\right) D_{\nu }\left( 2^{1/2}x^{1/2}p^{1/2}\right) = \nonumber \\
&\hspace{1cm}2^{-1/2}x^{-1/2}p^{-1/2}\exp \left( \tfrac{1}{2}p\left(
y+x\right) \right) D_{\mu }\left( 2^{1/2}y^{1/2}p^{1/2}\right) D_{\nu
+1}\left( 2^{1/2}x^{1/2}p^{1/2}\right)  \nonumber \\
&\hspace{1cm}+\nu 2^{-1/2}x^{-1/2}p^{-1/2}\exp \left( \tfrac{1}{2}p\left(
y+x\right) \right) D_{\mu }\left( 2^{1/2}y^{1/2}p^{1/2}\right) D_{\nu
-1}\left( 2^{1/2}x^{1/2}p^{1/2}\right)\nonumber 
\end{align}
\noindent Plugging the transform (\ref{ilt7}) into the latter expression and simplifying the result via the linear transformation formula (\ref{15.3.3})
gives (\ref{ilt8}).
\end{proof}
\begin{corollary}
The below derivations employ the following property of the Gaussian hypergeometric function%
\[
\mathstrut_2 F_1\left( 1,\dfrac{1}{2};2;z\right) =\mathstrut_2 F_1\left( \dfrac{1}{2},1;2;z\right) =\dfrac{%
2}{1+\sqrt{1-z}}
\]%
see Equation (84) on p. 473 in \cite{pbm390}.
\noindent Using $\mu =\nu =-1$ in (\ref{ilt8}) gives the following inverse Laplace transform for the product of two complementary error functions%
\begin{align}
&\exp \left( p\left( x+y\right) \right) \ \operatorname{erfc}\left(
y^{1/2}p^{1/2}\right) \ \operatorname{erfc}\left( x^{1/2}p^{1/2}\right) = \label{ilterf2} \\
&\hspace{1cm}\dfrac{1}{\pi }\int_{0}^{\infty }\exp \left( -pt\right) \dfrac{\sqrt{x}%
\sqrt{x+t}+\sqrt{y}\sqrt{y+t}}{\left( x+y+t\right) \sqrt{\left( x+t\right)
\left( y+t\right) }}dt \nonumber \\
&\hspace{2cm}[\operatorname{Re}p>0,\left\vert \arg y\right\vert <\pi,y\geqslant0,\left\vert \arg x\right\vert <\pi,x\geqslant0,\left\vert \arg x + \arg y\right\vert <\pi ] \nonumber
\end{align}%
Using $p=1$, $y^{1/2}=a$ and  $x^{1/2}=b$ then gives the following integral
representation for the product of two complementary error functions%
\begin{align}
&\operatorname{erfc}\left( a\right) \ \operatorname{erfc}\left( b\right) = \label{intrep2} \\
&\hspace{1cm}\dfrac{1}{\pi }\exp \left( -\left( a^{2}+b^{2}\right) \right)
\int_{0}^{\infty }\exp \left( -t\right) \dfrac{a\sqrt{t+a^{2}}+b\sqrt{t+b^{2}%
}}{\left( t+a^{2}+b^{2}\right) \sqrt{\left( t+a^{2}\right) \left(
t+b^{2}\right) }}dt \nonumber \\
&\hspace{2cm}[\operatorname{Re}a>0, a\geqslant 0, \operatorname{Re}b>0, b\geqslant 0]\nonumber
\end{align}%
which gives an alternative to the representation given on p. 70 in \cite{v16}. Using $a=0$ and $\operatorname{erfc}\left( 0\right) =1$, see Equation (40:7) in \cite{oms09}, gives %
\begin{align}
&\operatorname{erfc}\left( b\right) =\dfrac{b}{\pi }\exp \left( -b^{2}\right)
\int_{0}^{\infty }\dfrac{\exp \left( -t\right) }{\left( t+b^{2}\right) \sqrt{%
t}}dt \nonumber \\
&\hspace{2cm}[\operatorname{Re}b>0, b>0] \nonumber \\
&\operatorname{erf}\left( b\right) =1-\dfrac{b}{\pi }\exp \left( -b^{2}\right)
\int_{0}^{\infty }\dfrac{\exp \left( -t\right) }{\left( t+b^{2}\right) \sqrt{%
t}}dt \label{intrep3} \\
&\hspace{2cm}[\operatorname{Re}b>0, b>0] \nonumber
\end{align}%
The definition of the complementary error function gives $\operatorname{erf}(a)\operatorname{erf}(b)=\operatorname{erf}(b)-\operatorname{erfc}(a)\operatorname{erf}(b)$ such that plugging (\ref{intrep3}) and (\ref{intrep1}) into the latter relation gives%
\begin{align}
&1-\operatorname{erf}(a)\operatorname{erf}(b)= \label{intrep4} \\
&\hspace{1cm}\dfrac{b}{\pi }\exp \left( -b^{2}\right) \int_{0}^{\infty }\exp \left(
-t\right) \left\{ \dfrac{1}{\left( t+b^{2}\right) \sqrt{t}}-\dfrac{\exp
\left( -a^{2}\right) }{\left( t+a^{2}+b^{2}\right) \sqrt{t+a^{2}}}\right\} dt \nonumber
\\
&\hspace{1cm}+\dfrac{a}{\pi }\exp \left( -a^{2}\right) \int_{0}^{b^{2}}\dfrac{\exp
\left( -t\right) }{\left( t+a^{2}\right) \sqrt{t}}dt \nonumber \\
&\hspace{2cm}[\operatorname{Re}a>0, a>0, \operatorname{Re}b>0, b>0] \nonumber
\end{align}%
\noindent which generalizes the expression for 1--erf$\left(a\right)^{2}$ in Equation (8) on p. 4 in \cite{ng69} to differing arguments. Note that the representation in \cite{ng69} can easily be obtained from (\ref{intrep4}) by using $a=b$ which gives
\[
1-\operatorname{erf}(a)^{2}=\dfrac{2a}{\pi }\exp \left( -a^{2}\right)
\int_{0}^{a^{2}}\dfrac{\exp \left( -t\right) }{\left( t+a^{2}\right) \sqrt{t}%
}dt
\]%
\noindent The substitution $t=a^{2}s^{2}$ then gives%
\[
1-\operatorname{erf}(a)^{2}=\dfrac{4}{\pi }\exp \left( -a^{2}\right) \int_{0}^{1}%
\dfrac{\exp \left( -a^{2}s^{2}\right) }{\left( s^{2}+1\right) }ds
\]%
which is the integral representation in \cite{ng69}.
\end{corollary}
\bigskip
\section{Correcting two inverse Laplace transforms} 
\noindent This Section utilizes the above results to correct two inverse Laplace transforms that are frequently found. 
\medskip
\subsection{First correction}%
The following inverse Laplace transform is specified in Equation (3.11.4.3) in \cite{pbm592}
\begin{align*}
&D_{\nu }\left( a\sqrt{p}\right) D_{-\nu -1}\left( a\sqrt{p}\right) =\\
&\hspace{1cm} \int_{a}^{\infty }\exp \left( -pt\right)\dfrac{\left( t^{2}-a^{2}\right) ^{-1/2}}{\sqrt{2t}}\cos %
\left[ \left( \nu +\dfrac{1}{2}\right) \arccos \left[ \dfrac{a^{2}}{2t}%
\right] \right] dt\hspace{0.4cm} \ast\ast
\end{align*}
\noindent where $\ast\ast$ indicates that the expression is not correct. The corrected expression, however, can easily be obtained from the results in Section $3$.
\medskip
\begin{theorem}
Let $\nu$ be a complex number. Then, the following inverse Laplace transform holds for $\operatorname{Re}p>0$, $\operatorname{Re}a>0$, $a>0$
\begin{align}
&D_{\nu }\left( a\sqrt{p}\right) D_{-\nu -1}\left( a\sqrt{p}\right) = \label{ilt9}\\
&\hspace{1cm}\int_{\tfrac{1}{2}a^{2}}^{\infty }\exp \left( -pt\right) \dfrac{a\left(
t^{2}-\tfrac{a^{4}}{4}\right) ^{-1/2}}{\sqrt{2\pi t}}\cos \left[ \left( 2\nu
+1\right) \arcsin \left[ \sqrt{\dfrac{2t-a^{2}}{4t}}\right] \right] dt \nonumber
\end{align}
\end{theorem}
\begin{proof}
Using $a=2^{1/2}x^{1/2}=2^{1/2}y^{1/2}\ $and $\mu =-\nu -1$ allows to rewrite (\ref{ilt8}) as follows%
\begin{align*}
&\exp \left( \tfrac{1}{2}a^{2}p\right) D_{\nu }\left( a\sqrt{p}\right)
D_{-\nu -1}\left( a\sqrt{p}\right) = \\
&\hspace{1cm}\dfrac{1}{a\sqrt{\pi }}\int_{0}^{\infty }\exp \left( -pt\right)
t^{-1/2}\left\{ \mathstrut_2 F_1\left( -\dfrac{1+\nu }{2},\dfrac{1+\nu }{2};\dfrac{1}{2};%
\dfrac{4t\left( a^{2}+t\right) }{\left( a^{2}+2t\right) ^{2}}\right) \right. 
\\
&\hspace{1cm}\left. +\dfrac{2\nu t}{a^{2}+2t}\mathstrut_2 F_1\left( \dfrac{1-\nu }{2},\dfrac{1+\nu }{2}%
;\dfrac{3}{2};\dfrac{4t\left( a^{2}+t\right) }{\left( a^{2}+2t\right) ^{2}}%
\right) \right\} dt
\end{align*}
\noindent Multiplying both sides by $\exp \left( -\tfrac{1}{2}a^{2}p\right)$, using the substitution $s=t+\tfrac{1}{2}a^{2}$ and subsequently re-introducing $t$ gives
\begin{align*}
&D_{\nu }\left( a\sqrt{p}\right) D_{-\nu -1}\left( a\sqrt{p}\right) = \\
&\hspace{1cm}\dfrac{2^{1/2}}{a\sqrt{\pi }}\int_{\tfrac{1}{2}a^{2}}^{\infty }\exp \left(
-pt\right) \left( 2t-a^{2}\right) ^{-1/2}\left\{ \mathstrut_2 F_1\left( -\dfrac{1+\nu }{2},%
\dfrac{1+\nu }{2};\dfrac{1}{2};\dfrac{4t^{2}-a^{4}}{4t^{2}}\right) \right. 
\\
&\hspace{1cm}\left. +\dfrac{\nu \left( 2t-a^{2}\right) }{2t}\mathstrut_2 F_1\left( \dfrac{1-\nu }{2},%
\dfrac{1+\nu }{2};\dfrac{3}{2};\dfrac{4t^{2}-a^{4}}{4t^{2}}\right) \right\} dt
\end{align*}
\noindent The quadratic transformation formula in Equation (15.3.22) in \cite{as72} states
\[
\mathstrut_2 F_1\left( a,b;a+b+\dfrac{1}{2};z\right) =\mathstrut_2 F_1\left( 2a,2b;a+b+\dfrac{1}{2};\dfrac{%
1}{2}-\dfrac{1}{2}\sqrt{1-z}\right) 
\]
\noindent Using the latter relation gives
\begin{align*}
&D_{\nu }\left( a\sqrt{p}\right) D_{-\nu -1}\left( a\sqrt{p}\right) = \\
&\hspace{1cm}\dfrac{2^{1/2}}{a\sqrt{\pi }}\int_{\tfrac{1}{2}a^{2}}^{\infty }\exp \left(
-pt\right) \left( 2t-a^{2}\right) ^{-1/2}\left\{ \mathstrut_2 F_1\left( -1-\nu ,1+\nu ;%
\dfrac{1}{2};\dfrac{2t-a^{2}}{4t}\right) \right.  \\
&\hspace{1cm}\left. +\dfrac{\nu \left( 2t-a^{2}\right) }{2t}\mathstrut_2 F_1\left( 1-\nu ,1+\nu ;%
\dfrac{3}{2};\dfrac{2t-a^{2}}{4t}\right) \right\} dt
\end{align*}

\noindent The latter result can be simplified on the basis of the relations (15.2.10) and (15.2.20) in \cite{as72}, respectively
\begin{align*}
&\left( c-a\right) \mathstrut_2 F_1\left( a-1,b;c;z\right) +\left( 2a-c-az+bz\right)
\mathstrut_2 F_1\left( a,b;c;z\right) \\
&\hspace{2cm}+a\left( z-1\right) \mathstrut_2 F_1\left( a+1,b;c;z\right) =0 \\
&c\left( 1-z\right) \mathstrut_2 F_1\left( a,b;c;z\right) -c\mathstrut_2 F_1\left( a-1,b;c;z\right)
+\left( c-b\right) z\mathstrut_2 F_1\left( a,b;c+1;z\right) =0
\end{align*}
The latter two relations can be combined into
\begin{align*}
&\left( ac-c^{2}\right) \mathstrut_2 F_1\left( a-1,b;c;z\right) +\left( c^{2}-ac+c\left(
a-b\right) z\right) \mathstrut_2 F_1\left( a,b;c;z\right) \\
&\hspace{2cm} +a\left( b-c\right) z\mathstrut_2 F_1\left(
a+1,b;c+1;z\right) =0
\end{align*}

\noindent which gives
\begin{align*}
&\dfrac{a^{2}}{2t}\mathstrut_2 F_1\left( 1+\nu,-\nu;\dfrac{1}{2};\dfrac{2t-a^{2}}{4t}\right) 
=\mathstrut_2 F_1\left( -1-\nu,1+\nu;\dfrac{1}{2};\dfrac{2t-a^{2}}{4t}\right)  \\
&\hspace{2cm}+\dfrac{\nu\left( 2t-a^{2}\right) }{2t}\mathstrut_2 F_1\left( 1-\nu,1+\nu;\dfrac{3}{2};\dfrac{%
2t-a^{2}}{4t}\right)
\end{align*}
\noindent This allows to rewrite the inverse Laplace transform as
\begin{align*}
&D_{\nu}\left( a\sqrt{p}\right) D_{-\nu-1}\left( a\sqrt{p}\right) = \\
&\hspace{1.5cm}\dfrac{a}{\sqrt{2\pi} }\int_{\tfrac{1}{2}a^{2}}^{\infty }\exp \left( -pt\right)\frac{\left(
2t-a^{2}\right) ^{-1/2}}{t} \mathstrut_2 F_1\left( 1+\nu,-\nu;
\dfrac{1}{2};\dfrac{2t-a^{2}}{4t}\right) dt
\end{align*}
\noindent Equation (90) on p. 460 in \cite{pbm390} states%
\[
\mathstrut_2 F_1\left( a,1-a;\dfrac{1}{2};z\right) =\mathstrut_2 F_1\left( 1-a,a;\dfrac{1}{2};z\right) =%
\dfrac{1}{\sqrt{1-z}}\cos \left[ \left( 2a-1\right) \arcsin \left[ \sqrt{z}%
\right] \right] 
\]
\noindent Employing the latter property then gives (\ref{ilt9}).
\end{proof}
\medskip
\subsection{Second correction}
The following inverse Laplace transform can be found in Equation (11) on p. 218 in \cite{emoti154}, Equation (16.7) on p. 379 in \cite{ob73} and Equation (3.11.5.1) in \cite{pbm592}
\begin{align*}
&\exp \left( \tfrac{1}{4}a^{2}p^{2}\right) D_{\mu }\left( ap\right) D_{\nu
}\left( ap\right) = \\
&\hspace{1.5cm}\dfrac{1}{\Gamma \left[ -\mu -\nu \right] }\int_{0}^{\infty }\exp \left(
-pt\right) a^{\mu +\nu }t^{-\left( 1+\mu +\nu \right) }\exp \left( -\dfrac{%
t^{2}}{2a^{2}}\right)\\
&\hspace{1.5cm}\times \mathstrut_2 F_2\left( -\mu ,-\nu ;-\dfrac{\mu +\nu }{2},\dfrac{1-\mu
-\nu }{2};\dfrac{t^{2}}{4a^{2}}\right) dt\hspace{0.4cm}\ast \ast
\end{align*}
\begin{theorem}
Let $\nu$ and $\mu$ be two complex numbers with $\operatorname{Re}\left(\mu +\operatorname\nu\right) <0$. Then, the following inverse Laplace transform holds for $\operatorname{Re}p>0$, $\operatorname{Re}a>0$, $a>0$
\begin{align}
&\exp \left( \tfrac{1}{2}a^{2}p^{2}\right) D_{\mu }\left( ap\right) D_{\nu
}\left( ap\right) = \label{ilt10} \\
&\hspace{1.5cm}\dfrac{1}{\Gamma \left[ -\mu -\nu \right] }\int_{0}^{\infty }\exp \left(
-pt\right) a^{\mu +\nu }t^{-\left( 1+\mu +\nu \right) }\exp \left( -\dfrac{%
t^{2}}{2a^{2}}\right)\nonumber \\
&\hspace{1.5cm}\times \mathstrut_2 F_2\left( -\mu ,-\nu ;-\dfrac{\mu +\nu }{2},\dfrac{1-\mu
-\nu }{2};\dfrac{t^{2}}{4a^{2}}\right) dt\nonumber
\end{align}
\end{theorem}
\begin{proof}
From the specification of, for instance, the inverse Laplace transform (\ref{ilt8}), it is clear that the left-hand side of the expression in \cite{emoti154,ob73,pbm592} contains a misprint as the exponential term should be $\exp \left( \tfrac{1}{2}a^{2}p^{2}\right)$ rather than $\exp \left( \tfrac{1}{4}a^{2}p^{2}\right)$.
\end{proof}
\medskip
\section{Two new definite integrals for the generalized hypergeometric function} 
\noindent The below definite integrals for the generalized hypergeometric function are derived from the inverse Laplace transform (\ref{ilt10}) in combination with two results from Section $3$.

\subsection{First integral} Using $a=2^{1/2}x^{1/2}$ in (\ref{ilt10}) gives
\begin{align}
&\exp \left( p^{2}x\right) D_{\mu }\left( 2^{1/2}x^{1/2}p\right) D_{\nu
}\left( 2^{1/2}x^{1/2}p\right) = \label{int1}\\
&\hspace{1.5cm}\dfrac{\left( 2x\right) ^{\left( \mu +\nu \right) /2}}{\Gamma \left[ -\mu
-\nu \right] }\int_{0}^{\infty }\exp \left( -pt\right) t^{-\left( 1+\mu +\nu
\right) }\exp \left( -\dfrac{t^{2}}{4x}\right)\nonumber\\
&\hspace{1.5cm}\times \mathstrut_2 F_2\left( -\mu ,-\nu ;-\dfrac{%
\mu +\nu }{2},\dfrac{1-\mu -\nu }{2};\dfrac{t^{2}}{8x}\right) dt\nonumber
\end{align}
\noindent and the inverse Laplace transform (\ref{ilt8}) for $y=x$ is
\begin{align}
&\exp \left(px \right) D_{\mu }\left(
2^{1/2}x^{1/2}p^{1/2}\right) D_{\nu }\left( 2^{1/2}x^{1/2}p^{1/2}\right) = \label{int1b}\\
&\hspace{1.5cm}\dfrac{2^{\left( \mu +\nu \right) /2}x^{-1/2}}{\Gamma \left[ -\left( \mu
+\nu \right) /2\right] }\int_{0}^{\infty}\exp \left( -pt\right) t^{-1-\left( \nu
+\mu \right) /2}\left( x+t\right) ^{\left(1+\mu+\nu \right) /2} \nonumber \\
&\hspace{1.5cm}\times \left\{ \mathstrut_2 F_1\left( -\dfrac{%
\mu }{2},-\dfrac{1+\nu }{2};-\dfrac{\mu +\nu }{2};\dfrac{t\left(
2x+t\right) }{\left( x+t\right)^2 }\right) \right. \nonumber  \\
&\hspace{1.5cm}\left. -\dfrac{\nu t}{\left(\mu +\nu\right)\left(x+t\right) }\mathstrut_2 F_1\left( -\dfrac{\mu }{2},%
\dfrac{1-\nu }{2};1-\dfrac{\mu +\nu }{2};\dfrac{t\left( 2x+t\right) }{%
\left( x+t\right)^2 }\right) \right\} dt\nonumber 
\end{align}

\noindent Let $f\left( t\right) $ be the original function in the Laplace transform (\ref{int1}) and $%
F\left( p\right) $ be the corresponding image function. Equation (26) on p. 4
of \cite{pbm592} states that the original function of the image function $F\left(
p^{1/2}\right)$ then is related to $f\left( t\right) $ as follows
\begin{equation}
\dfrac{1}{2\sqrt{\pi t^{3}}}\int_{0}^{\infty }\tau \exp \left( -\dfrac{\tau
^{2}}{4t}\right) f\left( \tau \right) d\tau \label{26pbm592}
\end{equation}

\noindent Hence, plugging the original function for the inverse Laplace transform (\ref{int1}) into the expression (\ref{26pbm592}) gives the original function of expression (\ref{int1b}). Straightforward simplifications and redefinitions of variables then give the following definite integral for the generalized hypergeometric function%
\begin{align}
&\int_{0}^{\infty }t^{-\left( \mu +\nu \right) }\exp \left( -\dfrac{x+y}{4xy%
}t^{2}\right) \mathstrut_2 F_2\left( -\mu ,-\nu ;-\dfrac{\mu +\nu }{2},\dfrac{1-\mu -\nu }{2%
};\dfrac{t^{2}}{8x}\right) dt= \nonumber \\
&2^{-\left( \mu +\nu \right) }\Gamma \left[ \dfrac{1-\mu -\nu}{2} \right] y\left( \dfrac{x+y}{xy}\right) ^{\left( 1+\mu +\nu \right)
/2}\left\{ \mathstrut_2 F_1\left( -\dfrac{\mu }{2},-\dfrac{1+\nu }{2};-\dfrac{\mu +\nu }{2};%
\dfrac{y\left( 2x+y\right) }{\left( x+y\right) ^{2}}\right) \right.  \nonumber \\
&\left. -\dfrac{\nu y }{\left(\mu +\nu \right)\left(x+y\right)}\mathstrut_2 F_1\left( -\dfrac{%
\mu }{2},\dfrac{1-\nu }{2};1-\dfrac{\mu +\nu }{2};\dfrac{y\left( 2x+y\right) 
}{\left( x+y\right) ^{2}}\right) \right\} \label{int2} \\
& \hspace{1.5cm}\left[ \operatorname{Re}\left( \mu + \nu \right) < 1, \operatorname{Re}x>0, x>0,\operatorname{Re}y>0, y>0 \right] \nonumber
\end{align}
\medskip
\subsection{Second integral}
\noindent Again, let $f\left( t\right) $ be the original function in the Laplace transform (\ref{int1}) and $%
F\left( p\right) $ be the corresponding image function. Equation (29) on p. 5
of \cite{pbm592} states that the original function of the image function $%
p^{-1/2}F\left( p^{1/2}\right) $ is given by%
\begin{equation}
\dfrac{1}{\sqrt{\pi t}}\int_{0}^{\infty }\exp \left( -\dfrac{\tau ^{2}}{4t}%
\right) f\left( \tau \right) d\tau \label{29pbm592}
\end{equation}
\noindent The property in (\ref{29pbm592}) establishes a relation between the inverse Laplace transforms for $\exp \left( p^{2}x\right) D_{\mu
}\left( 2^{1/2}x^{1/2}p\right)D_{\nu }\left( 2^{1/2}x^{1/2}p\right)$ as well as $p^{-1/2}\exp \left( px\right) D_{\mu }\left(2^{1/2}x^{1/2}p^{1/2}\right)$  $D_{\nu }\left( 2^{1/2}x^{1/2}p^{1/2}\right) $. Equation (\ref{29pbm592}) then allows to obtain the following indefinite integral%
\begin{align}
&\int_{0}^{\infty }t^{-\left( 1+\mu +\nu \right) }\exp \left( -\dfrac{x+y}{%
4xy}t^{2}\right)  \mathstrut_2 F_2\left( -\mu ,-\nu ;-\dfrac{\mu +\nu }{2},\dfrac{1-\mu -\nu 
}{2};\dfrac{t^{2}}{8x}\right) dt= \nonumber \\
&2^{-\left( 1+\mu +\nu \right) }\Gamma \left[ -\dfrac{\mu +\nu}{2}%
\right] \left( \dfrac{x+y}{xy}\right) ^{\left( \mu +\nu \right) /2} \mathstrut_2 F_1\left( -%
\dfrac{\mu }{2},-\dfrac{\nu }{2};\dfrac{1-\mu -\nu }{2};\dfrac{y\left(
2x+y\right) }{\left( x+y\right) ^{2}}\right) \label{int3} \\
& \hspace{1.5cm}\left[ \operatorname{Re}\left( \mu + \nu \right) < 0, \operatorname{Re}x\geqslant 0, x>0, \operatorname{Re}y\geqslant 0, y>0\right] \nonumber 
\end{align}
\\
\\


\begin{thebibliography}{99}

\bibitem{kcw96} 
Yu. P. Kalmykov, W. T. Coffey \and J. T. Waldron (1996), Exact analytic solution for the correlation time of a Brownian particle in a doublewell potential from the Langevin equation, J. Chem. Phys. \textbf{105}, 2112--2118.

\bibitem{dfg92} 
J. L. deLyra, S. K. Foong \and T. E. Gallivan (1992), Finite lattice systems with true critical behavior, Phys. Rev. D \textbf{46}, 1643--1657. 

\bibitem{zm07} 
T. V. Zaqarashvili \and K. Murawski (2007), Torsional oscillations of longitudinally inhomogeneous coronal loops, Astron. Astrophys. \textbf(470), 353--357.

\bibitem{nl19} 
Y. Nie \and V. Linetsky (2019), Sticky reflecting Ornstein--Uhlenbeck diffusions and the Vasicek interest rate model with the sticky zero lower bound, Stoch. Models, forthcoming.
 
\bibitem{cr71} 
R. M. Capocelli \and L. M. Ricciardi (1971), Diffusion Approximation and First Passage Time Problem for a Model Neuron, Kybernetik \textbf{8}, 214--223.

\bibitem{bv09} 
J.-F. Bercher \and C. Vignat (2009), On minimum Fisher information distributions with restricted support and fixed variance, Inform. Sci. \textbf{179}, 3832--3842.

\bibitem{cd83} 
R. Combescot \and T. Dombre (1983), Superfluid current in $^{3}$He-$A$ at $T=0$, Phys. Rev. B \textbf{28}, 5140--5149.

\bibitem{m99} 
C. Malyshev (1999), Higher corrections to the mass current in weakly inhomogeneous superfluid $^{3}$He-$A$, Phys. Rev. B \textbf{59}, 7064--7075.

\bibitem{d75} 
L. Durand (1975), Nicholson-type integrals for products of Gegenbauer functions and related topics, in: Theory and Applications of Special Functions (R. A. Askey, Ed.), Academic Press, New York, 353--374.

\bibitem{em99} 
\'{A}. Elbert \and M. E. Muldoon (1999), Inequalities and monotonicity properties for zeros of Hermite functions, Proc. Roy. Soc. Edinburgh Sect. A \textbf{129}, 57--75.

\bibitem{em08} 
\'{A}. Elbert \and M. E. Muldoon (2008), Approximations for zeros of Hermite functions, in: Special Functions and Orthogonal Polynomials
(D. Dominici \and R. S. Maier, Eds.), Contemporary Mathematics \textbf{471}, American Mathematical Society, Providence, 117--126.

\bibitem{emoti154} 
A. Erd\'{e}lyi, W. Magnus, F. Oberhettinger \and F. G. Tricomi (1954), Tables of Integral Transforms, Volume 1, McGraw-Hill, New York.

\bibitem{ob73} 
F. Oberhettinger \and L. Badii (1973), Tables of Laplace Transforms, Springer--Verlag, Berlin.

\bibitem{pbm592} 
A. P. Prudnikov, Yu. A. Brychkov \and O. I. Marichev (1992), Integrals and Series. Inverse Laplace Transforms, Volume 5, Gordon and Breach, New York.

\bibitem{ng69} 
E. W. Ng \and M. Geller (1969), A Table of Integrals of the Error Functions, J. Res. NBS \textbf{73B}, 1--20.

\bibitem{gr14} 
I. S. Gradshteyn \and I. M. Ryzhik (2014), Table of Integrals, Series, and Products (D. Zwillinger \and V. Moll, Eds.), Eighth Edition, Academic Press, New York.

\bibitem{w02} 
E. T. Whittaker (1902), On the Functions associated with the Parabolic Cylinder in Harmonic Analysis, Proc. Lond. Math. Soc. \textbf{35}, 417--427.

\bibitem{emoth253} 
A. Erd\'{e}lyi, W. Magnus, F. Oberhettinger \and F. G. Tricomi (1953), Higher Transcendental Functions, Volume 2, McGraw-Hill, New York.

\bibitem{oms09} 
K. Oldham, J. Myland \and J. Spanier (2009), An Atlas of Functions, Second Edition, Springer--Verlag, Berlin.
   
\bibitem{pbm492} 
A. P. Prudnikov, Yu. A. Brychkov \and O. I. Marichev (1992), Integrals and Series. Direct Laplace Transforms, Volume 4, Gordon and Breach, New York.

\bibitem{as72}     
M. Abramowitz \and I. A. Stegun (1972), Handbook of Mathematical Functions with Formulas, Graphs, and Mathematical Tables, Dover Publications, New York.
 
\bibitem{emoth153} 
A. Erd\'{e}lyi, W. Magnus, F. Oberhettinger \and F. G. Tricomi (1953), Higher Transcendental Functions, Volume 1, McGraw-Hill, New York.

\bibitem{bchl92} 
J. V. Beck, K. D. Cole, A. Haji-Sheikh \and B. Litkouhi (1992), Heat Conduction Using Green's Functions, Hemisphere Publishing Corporation, London.

\bibitem{v17} 
D. Veestraeten (2017), An integral representation for the product of parabolic cylinder functions, Integr. Transf. Spec. F. \textbf{28}, 15--21.

\bibitem{pbm390} 
A. P. Prudnikov, Yu. A. Brychkov \and O. I. Marichev (1990), Integrals and Series. More Special Functions, Volume 3, Gordon and Breach, New York.

\bibitem{v16} 
D. Veestraeten (2016), Some integral representations and limits for (products of) the parabolic cylinder function, Integr. Transf. Spec. F. \textbf{27}, 64--77.

\end{thebibliography}
\end{document}